\documentclass[a4paper,12pt,reqno]{amsart}
\usepackage{amsfonts}
\usepackage{amsmath}
\usepackage{amssymb}
\usepackage[a4paper]{geometry}
\usepackage{mathrsfs}
\usepackage{hyperref}
\renewcommand\eqref[1]{(\ref{#1})} 
\numberwithin{equation}{section}
\theoremstyle{plain}
\newtheorem{thm}{Theorem}[section]
\newtheorem{prop}[thm]{Proposition}
\newtheorem{cor}[thm]{Corollary}

\theoremstyle{definition}
\newtheorem{defn}[thm]{Definition}
\newtheorem{rem}[thm]{Remark}

\renewcommand{\wp}{\mathfrak S}
\newcommand{\Rn}{\mathbb R^{n}}

\begin{document}

   \title[Critical Hardy inequalities]
   {Critical Hardy inequalities}

\author[Michael Ruzhansky]{Michael Ruzhansky}
\address{
  Michael Ruzhansky:
  \endgraf
  Department of Mathematics
  \endgraf
  Imperial College London
  \endgraf
  180 Queen's Gate, London SW7 2AZ
  \endgraf
  United Kingdom
  \endgraf
  {\it E-mail address} {\rm m.ruzhansky@imperial.ac.uk}
  }
\author[Durvudkhan Suragan]{Durvudkhan Suragan}
\address{
  Durvudkhan Suragan:
  \endgraf
  Institute of Mathematics and Mathematical Modelling
  \endgraf
  125 Pushkin str.
  \endgraf
  050010 Almaty
  \endgraf
  Kazakhstan
  \endgraf
  and
  \endgraf
  Department of Mathematics
  \endgraf
  Imperial College London
  \endgraf
  180 Queen's Gate, London SW7 2AZ
  \endgraf
  United Kingdom
  \endgraf
  {\it E-mail address} {\rm d.suragan@imperial.ac.uk}
  }

\thanks{The authors were supported in parts by the EPSRC
 grant EP/K039407/1 and by the Leverhulme Grant RPG-2014-02,
 as well as by the MESRK grant 5127/GF4. No new data was collected or generated during 
 the course of research.}


     \begin{abstract}
     We prove a range of critical Hardy inequalities and
     uncertainty type principles on one of most general subclasses
     of nilpotent Lie groups, namely the class of homogeneous groups.
     Moreover, we establish a new type of critical Hardy inequality 
     and prove Hardy-Sobolev type inequalities. Most of the obtained estimates 
     are new already for the case of $\mathbb R^{n}$. 
     For example, for any $f\in C_{0}^{\infty}(\mathbb{R}^{n}\backslash\{0\})$ our results imply the range of critical Hardy inequalities of the form
     $$
\qquad  \underset{R>0}{\sup}\left\|\frac{f-f_{R}}{|x|^{\frac{n}{p}}{\log}\frac{R}{|x|}}
\right\|_{L^{p}(\mathbb{R}^{n})}\leq \frac{p}{p-1}\left\| \frac{1}{|x|^{\frac{n}{p}-1}} \nabla f\right\|_{L^{p}(\mathbb{R}^{n})},\quad 1<p<\infty,
$$
where $f_{R}=f(R\frac{x}{|x|})$, with sharp constant $\frac{p}{p-1}$, recovering the known cases of $p=n$ and $p=2$. Moreover, our results also imply a new type of a critical Hardy inequality of the form
$$\left\|\frac{f}{|x|}\right\|_{L^{n}(\mathbb{R}^{n})}
	\leq n
	\left\|(\log|x|)\nabla f\right\|_{L^{n}(\mathbb{R}^{n})},
$$
	for all $f\in C_{0}^{\infty}(\mathbb{R}^{n}\backslash\{0\}),$ where the constant $n$ is sharp.
	However, homogeneous groups provide a perfect degree of generality to talk about such estimates without using specific properties of $\mathbb R^n$ or of the Euclidean distance.
  \end{abstract}

     \maketitle

\section{Introduction}

In a very short note \cite{Hardy:1920} related to the convergence of different series related to an inequality by D. Hilbert, G. H. Hardy proved an inequality (in one variable) now bearing his name, which in the multidimensional case of $\Rn$ can be formulated as

\begin{equation}\label{HRn}
\left\|\frac{f(x)}{|x|}\right\|_{L^{2}(\Rn)} \leq \frac{2}{n-2}\left\| \nabla f\right\|_{L^{2}(\Rn)},\quad n\geq 3,
\end{equation}
where $\nabla$ is the standard gradient in $\mathbb{R}^{n}$, $f\in C_{0}^{\infty}(\mathbb{R}^{n}\backslash \{0\})$,
and the constant $\frac{2}{n-2}$ is known to be sharp.
Nowadays this inequality plays an important role in many areas such as the spectral theory, geometric estimates, the theory of partial differential equations, associated to the Laplacian,
see e.g. \cite{Davies}, \cite{Brez2} or \cite{BEL} for reviews of the subject.
In the modern analysis of the $p$-Laplacian and for many other problems, the $L^{p}$-version of \eqref{HRn} takes the form
\begin{equation}\label{HRn-p}
\left\|\frac{f(x)}{|x|}\right\|_{L^{p}(\Rn)} \leq \frac{p}{n-p}\left\| \nabla f\right\|_{L^{p}(\Rn)},\quad
n\geq 2,\; 1\leq p<n,
\end{equation}
again with the sharp constant $\frac{p}{n-p}$.

We note historically, that for sequences \eqref{HRn-p} was first given as an analogue of Hardy's original version of \eqref{HRn} in an immediate response by Marcel Riesz to Hardy's demonstration of his inequality to Riesz, see \cite{Hardy:1920} for an account of this. 

In the critical case $p=n$ the inequality \eqref{HRn-p} fails for any constant. In bounded domains, however, the inequality
\begin{equation}\label{HRn-cr}
\left\|\frac{f(x)}{|x|(1+\log\frac{1}{|x|})}\right\|_{L^{n}(B)} \leq \frac{n}{n-1}\left\| \nabla f\right\|_{L^{n}(B)},\quad
n\geq 2,
\end{equation}
is valid, with sharp constant $\frac{n}{n-1}$, see Edmunds and Triebel
\cite{ET-1999}, where such inequality was also shown to be equivalent to the critical case of the Sobolev-Lorentz inequality. However, \eqref{HRn-cr} is not invariant under scalings in the same way as \eqref{HRn} and \eqref{HRn-p} are, and in Ioku, Ishiwata and Ozawa \cite{IIO} a global scaling invariant version of
\eqref{HRn-cr} was established:
\begin{equation}\label{HRn-sc}
\left\|\frac{f-f_{R}}{|x|\log\frac{R}{|x|}}\right\|_{L^{n}(\Rn)} \leq \frac{n}{n-1}\left\| \frac{x}{|x|}\cdot\nabla f\right\|_{L^{n}(\Rn)},\quad
n\geq 2,
\end{equation}
again with sharp constant $\frac{n}{n-1}$, $f_{R}=f(R\frac{x}{|x|})$.
We refer also to \cite{II} for the discussion of \eqref{HRn-cr} and \eqref{HRn-sc} in bounded domains. 
In the case of $\mathbb R^{n}$, we extend this inequality to the whole range of critical Hardy inequalities of the form 
\begin{multline}\label{LH2-Rn}
\qquad  \underset{R>0}{\sup}\left\|\frac{f-f_{R}}{|x|^{\frac{n}{p}}{\log}\frac{R}{|x|}}
\right\|_{L^{p}(\mathbb{R}^{n})}\leq
\frac{p}{p-1}\left\| \frac{1}{|x|^{\frac{n}{p}-1}} \frac{x}{|x|}\cdot\nabla f\right\|_{L^{p}(\mathbb{R}^{n})}
\\
\leq \frac{p}{p-1}\left\| \frac{1}{|x|^{\frac{n}{p}-1}} \nabla f\right\|_{L^{p}(\mathbb{R}^{n})},\quad 1<p<\infty,
\end{multline}
with the sharp constant.
In addition, among other things, compared to \eqref{HRn-cr}, we show another type of critical Hardy inequality, namely,
\begin{equation}\label{CKNl}
	\left\|\frac{f}{|x|}\right\|_{L^{n}(\mathbb{R}^{n})}
	\leq n
	\left\|\log|x|\nabla f\right\|_{L^{n}(\mathbb{R}^{n})},
	\end{equation}
	for all $f\in C_{0}^{\infty}(\mathbb{R}^{n}\backslash\{0\}),$ where the constant is sharp.

In the case of 
homogeneous Carnot groups inequalities of this type have been also intensively investigated. In this case inequality \eqref{HRn} takes the form
\begin{equation}\label{HG}
\left\|\frac{f(x)}{d(x)}\right\|_{L^{2}(\mathbb G)} \leq \frac{2}{Q-2}\left\| \nabla_{H} f\right\|_{L^{2}(\mathbb G)},\quad Q\geq 3,
\end{equation}
where $Q$ is the homogeneous dimension of the homogeneous Carnot group $\mathbb G$, $\nabla_{H}$ is the horizontal gradient, and $d(x)$ is the so-called $\mathcal L$-gauge (a particular quasi-norm) obtained from the fundamental solution of the sub-Laplacian: $d(x)$ is such that $d(x)^{2-Q}$ is a constant multiple of Folland's \cite{Folland-FS} fundamental solution of $\mathcal L$.
Thus, in the case of the Heisenberg group \eqref{HG} was proved by Garofalo and Lanconelli
\cite{GL}, and its extension to $p\not=2$ was obtained by Niu, Zhang and Wang \cite{NZW-Hardy-p}, see also \cite{Adimurthi-Sekar}.
On general homogeneous Carnot groups weighted versions of \eqref{HG} are also known, see
Goldstein and Kombe \cite{GolKom}, and its further refinements including boundary terms over arbitrary domains were obtained by the authors in \cite{Ruzhansky-Suragan:Layers}.
Further remainder term formulae and horizontal versions of Hardy type inequalities can be found in \cite{Ruzhansky-Suragan:L2-CKN} and \cite{Ruzhansky-Suragan:horizontal}. We also shall note that the studies of the remainder terms in Hardy type inequalities has a long history, initiated by Brezis and Nirenberg in \cite{Brez3} and then Brezis and Lieb \cite{Brez1} for Hardy-Sobolev inequalities, Brezis and V\'azquez in \cite[Section 4]{Brez4}, with many subsequent works in this subject.

The $L^{p}$-versions of \eqref{HG} analogous to \eqref{HRn-p} in the form
\begin{equation}\label{HG-p}
\left\|\frac{f(x)}{d(x)}\right\|_{L^{p}(\mathbb G)} \leq \frac{p}{Q-p}\left\| \nabla_{H} f\right\|_{L^{p}(\mathbb G)},\quad Q\geq 3,\; 1<p<Q,
\end{equation}
are also known: on groups of Heisenberg type \cite{DGP-Hardy-potanal}, on polarisable groups \cite{GolKom}, see also \cite{DAmbrosio-Hardy}, and on general Carnot groups
\cite{Jin-Shen:Hardy-Rellich-AM-2011,Lian:Rellich}, see also \cite{Ciatti-Cowling-Ricci}.

In this paper we establish the global version of the critical Hardy inequality \eqref{HRn-cr}-\eqref{HRn-sc} on general homogeneous groups.
Homogeneous groups are Lie groups equipped with a family of
dilations compatible with the group law.
The abelian group $(\mathbb{R}^{n}; +)$, the Heisenberg group, homogeneous Carnot groups, stratified Lie groups, graded Lie groups are all special cases of the homogeneous groups. Homogeneous groups have proved to be a
natural setting to generalise many questions of the Euclidean analysis.
Indeed, having both the group and dilation structures allows one to introduce
many notions coming from the Euclidean analysis.
We note that homogeneous groups are nilpotent, and
the class of homogeneous groups gives almost the class of all nilpotent Lie groups but is not equal to it; an example of a nine-dimensional nilpotent Lie group that does not allow for any family of dilations was constructed by Dyer \cite{Dyer-1970}.
Thus, our methods of proofs are not specific to $\mathbb R^{n}$ and so it is convenient to work in the generality of homogeneous groups. However, already in the case of isotropic $\mathbb R^{n}$ some of our estimates are new while other results also provide new insights because of the freedom of the choice of a quasi-norm $|\cdot|$ which does not have to be the Euclidean norm.

Thus, in particular, we show that for a homogeneous group $\mathbb{G}$ of homogeneous dimension $Q\geq 2$ and any homogeneous quasi-norm $|\cdot|$
we have the inequality
\begin{equation}\label{iLH20}
\underset{R>0}{\sup}\left\|\frac{f-f_{R}}{|x|{\log}\frac{R}{|x|}}
\right\|_{L^{Q}(\mathbb{G})}\leq
\frac{Q}{Q-1}\left\|\mathcal{R} f\right\|_{L^{Q}(\mathbb{G})}
\end{equation}
for all $f\in C_{0}^{\infty}(\mathbb{G}\backslash \{0\})$, where $\mathcal{R}:=\frac{d}{d|x|}$ is the radial derivative.
Here we denote
$f_{R}=f(R\frac{x}{|x|})$ for $x\in\mathbb{G}$ and $R>0$,
$\nabla=(X_{1},\ldots,X_{n})$, where $\{X_{1},\ldots,X_{n}\}$ is a basis of the Lie algebra $\mathfrak{g}$ of $\mathbb{G}$, $A$ is a $n$-diagonal matrix
\begin{equation}\label{EQ:mA}
A={\rm diag}(\nu_{1},\ldots,\nu_{n}),
\end{equation}
where $\nu_{k}$ is the homogeneous degree of $X_{k}.$
The operator $\mathcal{R}$ appearing in
\eqref{iLH20} is homogeneous of order $-1$ and can be interpreted at the radial operator on the homogeneous group $\mathbb G$.

In fact, we will show more than the inequality \eqref{iLH20}.
Namely, we will show the whole range of inequalities
\begin{equation}\label{LH2p}
\qquad  \underset{R>0}{\sup}\left\|\frac{f-f_{R}}{|x|^{\frac{Q}{p}}{\log}\frac{R}{|x|}}
\right\|_{L^{p}(\mathbb{G})}\leq
\frac{p}{p-1}\left\| \frac{1}{|x|^{\frac{Q}{p}-1}}\mathcal{R} f\right\|_{L^{p}(\mathbb{G})},
\end{equation}
for all $1< p<\infty$, where $\mathcal{R}=\frac{d}{d|x|}$.  For $p=Q$ this is the same as \eqref{iLH20}. For $p=2$, in the abelian case $G=\mathbb R^{n}$ the estimate \eqref{LH2p} was shown by Machihara, Ozawa and Wadade in \cite{MOW:Hardy-Hayashi}. The abelian case of \eqref{LH2p} in principle can be obtained from \cite[Theorem 1.1]{MOW:Sobolev-Lorenz-Zygmund}, when $|\cdot|$ is the Euclidean norm.

Consequently, \eqref{LH2p} implies critical versions of uncertainty principles on homogeneous groups that will be given in Corollary \ref{Luncertainty}.

The inequalities \eqref{LH2p} are all critical with respect to the weight in the left hand side, namely, all the weights $|x|^{-\frac{Q}{p}}$ in $L^{p}$ give $|x|^{-Q}$ which is the critical order for the integrability at zero and at infinity.

The constant $\frac{p}{p-1}$ in \eqref{LH2p} is sharp but is in general unattainable (see \cite{II,IIO} for the Euclidean case). In fact, the proof of Theorem \ref{CHardy} yields also the precise form of the remainder, i.e. of the difference between two terms in \eqref{LH2p}.

One can readily see that the inequality \eqref{iLH20} extends the critical Hardy inequality \eqref{HRn-sc}.
Indeed, in the case of $\mathbb G=\mathbb{R}^{n}$,
we have $Q=n$ and if $|x|$ is the Euclidean norm, \eqref{iLH20}
reduces to \eqref{HRn-sc}.
Our proof of \eqref{LH2p} is based on the Euclidean ideas \cite{IIO} and \cite{MOW:Sobolev-Lorenz-Zygmund} combined with further analysis on homogeneous groups developed by Folland and Stein \cite{FS-Hardy}, and later in \cite{FR}.

Since we are dealing with general homogeneous groups there may be no gradation nor stratification available in its Lie algebra - therefore, the appearance of the full gradient
$\nabla=(X_{1},\ldots,X_{n})$ in \eqref{iLH20} or \eqref{LH2p} is rather natural. Moreover, we note that we do not have to rely on a particular gauge $d(x)$ as in \eqref{HG} and \eqref{HG-p}, but can use any homogeneous quasi-norm $|x|$ in \eqref{iLH20} and \eqref{LH2p}.
In fact, also in the stratified case, one does not always need to work with the gauge (a particular norm) coming from the fundamental solution, and we refer to \cite{Ruzhansky-Suragan:horizontal} for critical Hardy and other inequalities in the horizontal setting of stratified groups. We also refer to \cite{Ruzhansky-Suragan:squares} for Hardy and Rellich inequalities with boundary terms for general H\"ormander's sums of squares.

We also note that \eqref{iLH20} is the critical case of $L^{p}$-Hardy inequalities on homogeneous groups that were established in \cite{Ruzhansky-Suragan:Hardy-hom-Lp}, taking the form
\begin{equation}\label{EQ:Hardy-Lp}
\left\|\frac{f}{|x|}
\right\|_{L^{p}(\mathbb{G})}\leq
\frac{p}{Q-p}\left\|\mathcal{R} f\right\|_{L^{p}(\mathbb{G})},
\quad 1<p<Q.
\end{equation}

This inequality can be regarded as an extension of \eqref{HRn-p} and \eqref{HG-p} to the setting of homogeneous groups, we refer to \cite{Ruzhansky-Suragan:Hardy-hom-Lp} for explanations. From this point of view the inequality \eqref{iLH20} is the critical case of \eqref{EQ:Hardy-Lp} with $p=Q$ in the same way as \eqref{HRn-sc} is the critical case of \eqref{HRn-p} with $p=n$.

Thus, the main aim of this paper is to obtain a critical Hardy inequality on
homogeneous groups generalising the known critical Hardy inequality of the Euclidean case as well as critical Hardy inequalities that would be already new also on $\mathbb R^{n}$. Moreover, such inequalities will also imply the corresponding versions of critical uncertainty principles on homogeneous groups. 

Furthermore, we establish a new type of critical Hardy inequality on general homogeneous groups, and also give improved versions of the inequality \eqref{HRn-p} for $p=2$ on quasi-balls of homogeneous (Lie) groups. This also generalises many previously known results on subclasses of nilpotent Lie groups. Such inequalities are also called Hardy-Sobolev type inequalities.

In fact, the Euclidean case of our homogeneous group results says that for each $f\in C_{0}^{\infty}(B(0,R)\backslash\{0\}), \,B(0,R)=\{x\in\Rn,\quad |x|<R\},$ we have
\begin{equation}\label{IntroCritLQHardy}
\int_{B(0,R)}\frac{|f(x)|^{n}}{|x|^{n}}dx \leq n^{n}\int_{B(0,R)}|\log|x||^{n}|\nabla f|^{n}dx,
\end{equation}
with the optimal constant, which can serve as a new critical version of the Hardy inequality \eqref{HRn-p} (see also \eqref{CKNl}), while it is also generalised in Section \ref{Sec4} to arbitrary  quasi-balls of homogeneous (Lie) groups.

In Section \ref{SEC:2} we very briefly review the main concepts of homogeneous
groups and fix the notation. In Section \ref{Sec3} we prove critical Hardy inequalities and uncertainty type principles. In Section \ref{Sec2} a new critical Hardy inequality is given.
In Section \ref{Sec4} we discuss a class of improved Hardy-Sobolev type inequalities on quasi-balls of homogeneous (Lie) groups.

\section{Preliminaries}
\label{SEC:2}

Following Folland and Stein \cite{FS-Hardy} we briefly recall the main notions concerning homogeneous groups. We adopt the notation from \cite{FR} and refer to it for further details.

A family of dilations of a Lie algebra $\mathfrak{g}$
is a family of linear mappings of the form
$$D_{\lambda}={\rm Exp}(A {\rm ln}\lambda)=\sum_{k=0}^{\infty}
\frac{1}{k!}({\rm ln}(\lambda) A)^{k},$$
where $A$ is a diagonalisable linear operator on $\mathfrak{g}$
with positive eigenvalues,
and each $D_{\lambda}$ is a morphism of the Lie algebra $\mathfrak{g}$,
that is, a linear mapping
from $\mathfrak{g}$ to itself which respects the Lie bracket:
$$\forall X,Y\in \mathfrak{g},\, \lambda>0,\;
[D_{\lambda}X, D_{\lambda}Y]=D_{\lambda}[X,Y].$$

\begin{defn}\label{DEF:hom}
A homogeneous group is a connected simply connected Lie group whose
Lie algebra is equipped with dilations.
\end{defn}

Homogeneous groups are nilpotent, the exponential mapping is a global diffeomorphism from $\mathfrak g$ to $\mathbb G$, thus leading to the dilations on $\mathbb G$ which we continue to denote by $D_{\lambda}x$ or simply by $\lambda x$.

If $x$ denotes a point in $\mathbb{G}$ the Haar measure
is denoted by
$dx$. The Haar measure of
a measurable subset $S$ of $\mathbb{G}$ is
denoted by $|S|$.
One easily checks that
\begin{equation}
|D_{\lambda}(S)|=\lambda^{Q}|S| \quad {\rm and}\quad \int_{\mathbb{G}}f(\lambda x)
dx=\lambda^{-Q}\int_{\mathbb{G}}f(x)dx,
\end{equation}
where
$$Q = {\rm Tr}\,A.$$
The number $Q$ is larger (or equal) than the usual (topological) dimension of the group:
$$n = {\rm dim} \mathbb{G} \leq Q,$$
and may replace it for certain questions of analysis.
For this reason the number $Q$
is called the homogeneous dimension of $\mathbb{G}$.

\begin{defn}\label{DEF:norm} A homogeneous quasi-norm is
a continuous non-negative function
$$\mathbb{G}\ni x\mapsto |x|\in [0,\infty),$$
satisfying

\begin{itemize}
\item  (symmetric) $|x^{-1}| = |x|$ for all $x\in \mathbb{G}$,
\item (1-homogeneous) $|\lambda x|=\lambda |x|$ for all
$x\in \mathbb{G}$ and $\lambda >0$,
\item (definite) $|x|= 0$ if and only if $x=0$.
\end{itemize}

\end{defn}

Every homogeneous group $\mathbb{G}$ admits a homogeneous
quasi-norm that is smooth away from the unit element but we do not need it here.
The $|\cdot|$-ball centred at $x\in\mathbb{G}$ with radius $R > 0$
can be defined by
$$B(x,R):=\{y\in \mathbb{G}: |x^{-1}y|<R\}.$$
We also use notation
$$B^{c}(x,R):=\{y\in \mathbb{G}: |x^{-1}y|>R\}.$$
The following polar decomposition was established in \cite{FS-Hardy}
(see also \cite[Section 3.1.7]{FR}).

\begin{prop}\label{polarinteg}
Let $\mathbb{G}$
be a homogeneous group equipped with a homogeneous
quasi-norm $\mid\cdot\mid$. Then there is a (unique)
positive Borel measure $\sigma$ on the
unit sphere
\begin{equation}\label{EQ:sphere1}\wp:=\{x\in \mathbb{G}:\,|x|=1\},\end{equation}
such that for all $f\in L^{1}(\mathbb{G})$, we have
\begin{equation}\label{EQ:sphere}
\int_{\mathbb{G}}f(x)dx=\int_{0}^{\infty}
\int_{\wp}f(ry)r^{Q-1}d\sigma(y)dr.
\end{equation}
\end{prop}

We now fix a basis $\{X_{1},...,X_{n}\}$ of $\mathfrak{g}$
such that
$$AX_{k}=\nu_{k}X_{k}$$
for each $k$, so that $A$ is given by \eqref{EQ:mA}. 
Denote by 
$e(x)=(e_{1}(x),\ldots,e_{n}(x))$ the vector that is determined by
$$
{\exp}_{\mathbb{G}}^{-1}(x)=e(x)\cdot \nabla\equiv\sum_{j=1}^{n}e_{j}(x)X_{j}.
$$
Then one has
$$x={\exp}_{\mathbb{G}}\left(e_{1}(x)X_{1}+\ldots+e_{n}(x)X_{n}\right),$$
$$rx={\exp}_{\mathbb{G}}\left(r^{\nu_{1}}e_{1}(x)X_{1}+\ldots
+r^{\nu_{n}}e_{n}(x)X_{n}\right),$$
$$e(rx)=(r^{\nu_{1}}e_{1}(x),\ldots,r^{\nu_{n}}e_{n}(x)).$$
Thus, since $r>0$ is arbitrary, without loss of generality taking $|x|=1$, we obtain
\begin{equation}
\frac{d}{d|rx|}f(rx) =  \frac{d}{dr}f({\exp}_{\mathbb{G}}
\left(r^{\nu_{1}}e_{1}(x)X_{1}+\ldots
+r^{\nu_{n}}e_{n}(x)X_{n}\right)).
\end{equation}
Denoting by
\begin{equation}\label{EQ:Euler}
\mathcal{R} :=\frac{d}{dr},
\end{equation}
for all $x\in \mathbb G$ this gives the equality
\begin{equation}\label{dfdr}
\frac{d}{d|x|}f(x)=\mathcal{R}f(x),
\end{equation}
for each homogeneous quasi-norm $|x|$ on a homogeneous group $\mathbb G.$
The formula \eqref{dfdr} plays a role of radial derivative on $\mathbb G$ and will be useful for our calculations.
We note that the operator
\begin{equation}\label{EQ:Euler}
{\tt Euler}:=|x|\mathcal{R}
\end{equation}
is homogeneous of order zero and can be thought of as the Euler operator on the homogeneous group $\mathbb G$ characterising the homogeneity of functions: a continuously differentiable function $f$ is positively homogeneous of order $\mu$, i.e. $f(rx)=r^{\mu}f(x)$ for all $r>0$ and $x\not=0$, if and only if ${\tt Euler} (f)=\mu f$.

\section{Critical Hardy inequalities and uncertainty type principle}
\label{Sec3}

We now present a range of critical Hardy inequalities on the homogeneous group $\mathbb{G}$. In the isotropic (standard) Euclidean case of $\mathbb R^{n}$ with the quasi-norm being the Euclidean norm, the following result was obtained in \cite{IIO} for $p=Q$ and in \cite{MOW:Hardy-Hayashi} for $p=2$. Thus, already in such setting the following inequalities for the whole range of $1<p<\infty$ are new.

\begin{thm}\label{CHardy}
Let $\mathbb{G}$ be a homogeneous group of homogeneous dimension $Q\geq 2$ and a homogeneous quasi-norm denoted by $|\cdot|$.
Let $f\in C_{0}^{\infty}(\mathbb{G}\backslash\{0\})$ and $f_{R}=f(R\frac{x}{|x|})$ for $x\in\mathbb{G}$ and $R>0$. Then the following generalised critical Hardy
inequality is valid:
\begin{equation}\label{LH2}
\qquad  \underset{R>0}{\sup}\left\|\frac{f-f_{R}}{|x|^{\frac{Q}{p}}{\log}\frac{R}{|x|}}
\right\|_{L^{p}(\mathbb{G})}\leq
\frac{p}{p-1}\left\| \frac{1}{|x|^{\frac{Q}{p}-1}} \mathcal{R} f\right\|_{L^{p}(\mathbb{G})},\quad 1<p<\infty,
\end{equation}
where
$\mathcal{R}$ is the radial operator on 
$\mathbb{G}$ with respect to the quasi-norm $|\cdot|$, and where the constant  $\frac{p}{p-1}$ is optimal.
\end{thm}

It follows from the remainder formula obtained in the proof (i.e. the exact expression \eqref{EQ:rem} for the difference between two sides of \eqref{LH2}) that the constant $\frac{p}{p-1}$ in
\eqref{LH2} is the best constant.

We note that for $p=Q$ the estimate \eqref{LH2} becomes
\begin{equation}\label{LH2-Q}
\qquad  \underset{R>0}{\sup}\left\|\frac{f-f_{R}}{|x|{\log}\frac{R}{|x|}}
\right\|_{L^{Q}(\mathbb{G})}\leq
\frac{Q}{Q-1}\left\|\mathcal{R} f\right\|_{L^{Q}(\mathbb{G})},\quad Q\geq 2,
\end{equation}
which gives the critical estimate \eqref{iLH20}.
In the Euclidean isotropic case inequalities \eqref{LH2} with the Euclidean norm imply
\eqref{LH2-Rn}.

\begin{proof}[Proof of Theorem \ref{CHardy}]
Introducing polar coordinates $(r,y)=(|x|, \frac{x}{\mid x\mid})\in (0,\infty)\times\wp$ on $\mathbb{G}$, where $|x|$ is a homogeneous quasi-norm of $x\in\mathbb{G}$ (see Definition \ref{DEF:norm}) and using Proposition \ref{polarinteg} one calculates
$$
\int_{B(0,R)}
\frac{|f(x)-f_{R}(x)|^{p}}
{|x|^{Q}|{\rm log}\frac{R}{|x|}|^{p}}dx
$$

$$
=\int_{0}^{R}\int_{\wp}
\frac{|f(ry)-f(Ry)|^{p}}
{r^{Q}\left({\rm log}\frac{R}{r}\right)^{p}}r^{Q-1}d\sigma(y)dr
$$

$$
=\int_{0}^{R}\frac{d}{dr}\left( \frac{1}{p-1}\frac{1}
{\left({\rm log}\frac{R}{r}\right)^{p-1}}\int_{\wp}
|f(ry)-f(Ry)|^{p}d\sigma(y)\right) dr
$$

$$
-\frac{p}{p-1} {\rm Re}\int_{0}^{R}\left( \frac{1}
{\left({\rm log}\frac{R}{r}\right)^{p-1}}\int_{\wp}
|f(ry)-f(Ry)|^{p-2}(f(ry)-f(Ry))
\overline{\frac{df(ry)}{dr}}d\sigma(y)\right) dr
$$
$$
=-\frac{p}{p-1} {\rm Re}\int_{0}^{R}\left( \frac{1}
{\left({\rm log}\frac{R}{r}\right)^{p-1}}\int_{\wp}
|f(ry)-f(Ry)|^{p-2}(f(ry)-f(Ry))
\overline{\frac{df(ry)}{dr}}d\sigma(y)\right) dr,
$$
where $\sigma$ is the Borel measure on
$\wp$ and the contribution
on the boundary at $r=R$ vanishes due to the inequalities
$$|f(ry)-f(Ry)|\leq C(R-r),$$
$$\frac{R-r}{R}\leq {\rm log}\frac{R}{r}.$$
Using the formula \eqref{dfdr}
we arrive at
$$
\int_{B(0,R)}
\frac{|f(x)-f_{R}(x)|^{p}}
{|x|^{Q}|{\rm log}\frac{R}{|x|}|^{p}}dx
$$
$$
=-\frac{p}{p-1} {\rm Re}\int_{0}^{R}\frac{1}
{\left({\rm log}\frac{R}{r}\right)^{p-1}}\int_{\wp}
|f(ry)-f(Ry)|^{p-2}
$$$$(f(ry)-f(Ry))
\overline{\frac{d f(ry)}{dr}}d\sigma(y) dr
$$
$$
=-\frac{p}{p-1} {\rm Re}\int_{B(0,R)}
\left|\frac{f(x)-f_{R}(x)}
{|x|^{\frac{Q}{p}}{\rm log}\frac{R}{|x|}}\right|^{p-2}
\frac{(f(x)-f_{R}(x))}{|x|^{\frac{Q}{p}}{\rm log}\frac{R}{|x|}}
\frac{1}{|x|^{\frac{Q}{p}-1}}\overline{\frac{df(x)}{d|x|}}dx.
$$
Similarly, one obtains
$$
\int_{B^{c}(0,R)}
\frac{|f(x)-f_{R}(x)|^{p}}
{|x|^{Q}|{\rm log}\frac{R}{|x|}|^{p}}dx
$$

$$
=\int_{R}^{\infty}\int_{\wp}
\frac{|f(ry)-f(Ry)|^{p}}
{r^{Q}\left({\rm log}\frac{r}{R}
\right)^{p}}r^{Q-1}d\sigma(y)dr
$$

$$
=-\int_{R}^{\infty}\frac{d}{dr}\left(\frac{1}{p-1} \frac{1}
{\left({\rm log}\frac{r}{R}\right)^{p-1}}\int_{\wp}
|f(ry)-f(Ry)|^{p}d\sigma(y)\right) dr
$$

$$
+\frac{p}{p-1} {\rm Re}\int_{R}^{\infty}\left( \frac{1}
{\left({\rm log}\frac{r}{R}\right)^{p-1}}\int_{\wp}
|f(ry)-f(Ry)|^{p-2}(f(ry)-f(Ry))
\overline{\frac{df(ry)}{dr}}d\sigma(y)\right) dr
$$
$$
=-\frac{p}{p-1} {\rm Re}\int_{B^{c}(0,R)}
\left|\frac{f(x)-f_{R}(x)}
{|x|^{\frac{Q}{p}}{\rm log}\frac{R}{|x|}}\right|^{p-2}
\frac{(f(x)-f_{R}(x))}{|x|^{\frac{Q}{p}}{\rm log}\frac{R}{|x|}}
\frac{1}{|x|^{\frac{Q}{p}-1}}\overline{\frac{df(x)}{d|x|}}dx.
$$
It follows that
$$\int_{\mathbb{G}}
\frac{|f(x)-f_{R}(x)|^{p}}
{|x|^{Q}|{\rm log}\frac{R}{|x|}|^{p}}dx=
-\frac{p}{p-1} {\rm Re}\int_{\mathbb{G}}
\left|\frac{f(x)-f_{R}(x)}
{|x|^{\frac{Q}{p}}{\rm log}\frac{R}{|x|}}\right|^{p-2}
$$$$
\frac{(f(x)-f_{R}(x))}{|x|^{\frac{Q}{p}}{\rm log}\frac{R}{|x|}}
\frac{1}{|x|^{\frac{Q}{p}-1}}\overline{\frac{df(x)}{d|x|}}dx$$
$$=\left(\frac{p}{p-1}\right)^{p}
\|v\|^{p}_{L^{p}(\mathbb{G})}-p
\int_{\mathbb{G}}I(u,-\frac{p}{p-1}v)\left|\frac{p}{p-1}v+u\right|^{2}dx,$$
where
$$u=\frac{f(x)-f_{R}(x)}
{|x|^{\frac{Q}{p}}{\rm log}\frac{R}{|x|}},$$
$$v=\frac{1}{|x|^{\frac{Q}{p}-1}}\frac{df(x)}{d|x|},$$
and $I$ is defined by
$$I(f,g):=\left(\frac{1}{p}|g|^{p}+\frac{1}{p^{\prime}}|f|^{p}-|f|^{p-2}\,{\rm Re}(f\overline{g})\right)|f-g|^{-2}\geq 0,\quad f\neq g,\quad\frac{1}{p}+\frac{1}{p^{\prime}}=1,$$
$$I(g,g):=\frac{p-1}{2}|g|^{p-2}.$$
That is,
\begin{equation}\label{EQ:rem}
\|u\|^{p}_{L^{p}(\mathbb{G})}=
\left(\frac{p}{p-1}\right)^{p}\|v\|^{p}_{L^{p}(\mathbb{G})}-p\int_{\mathbb{G}}I(u,-\frac{p}{p-1}v)\left|\frac{p}{p-1}v+u\right|^{2}dx.
\end{equation}
This proves the inequality \eqref{LH2} since the last term is non-positive.
\end{proof}

Theorem \ref{CHardy} implies the following
uncertainly type principles:

\begin{cor}[Uncertainly type principle on $\mathbb{G}$]
\label{Luncertainty}
Let $1< p<\infty$ and $f\in C_{0}^{\infty}(\mathbb{G}\backslash \{0\})$.
Then for any $R>0$ and $\frac{1}{p}+\frac{1}{q}=
\frac{1}{2}$ with $q>1$, we have
\begin{equation}\label{UP1}
\left\|\frac{1}{|x|^{\frac{Q}{p}-1}}\mathcal{R} f\right
\|_{L^{p}(\mathbb{G})}
\left\| f\right\|_{L^{q}(\mathbb{G})}\geq \frac{p-1}{p}
\left\|\frac{f(f-f_{R})}{|x|^{\frac{Q}{p}}{\log}\frac{R}{|x|}}\right
\|_{L^{2}(\mathbb{G})}
\end{equation}
and also
\begin{equation}\label{UP2}
\left\|\frac{1}{|x|^{\frac{Q}{p}-1}}\mathcal{R} f\right
\|_{L^{p}(\mathbb{G})}
\left\|\frac{f-f_{R}}{|x|^{\frac{Q}{p^{\prime}}}{\log}\frac{R}{|x|}}
\right\|_{L^{p^{\prime}}(\mathbb{G})}
\geq
\frac{p-1}{p}\left\|\frac{f-f_{R}}{|x|^{\frac{Q}{2}}
{\log}\frac{R}{|x|}}\right\|^{2}_{L^{2}(\mathbb{G})}
\end{equation}
for $\frac{1}{p}+\frac{1}{p^{\prime}}=1$.
\end{cor}

\begin{proof}[Proof of Corollary \ref{Luncertainty}]

By using the critical Hardy inequality in Theorem \ref{CHardy} we have
$$\left\|\frac{1}{|x|^{\frac{Q}{p}-1}}\mathcal{R} f\right\|_{L^{p}(\mathbb{G})}
\left\| f\right\|_{L^{q}(\mathbb{G})}\geq \frac{p-1}{p} \left\|\frac{f-f_{R}}{|x|^{\frac{Q}{p}}{\log}\frac{R}{|x|}}\right\|_{L^{p}(\mathbb{G})}
\left\| f\right\|_{L^{q}(\mathbb{G})}$$

$$=\frac{p-1}{p}\left(\int_{\mathbb{G}}\left|\frac{f-f_{R}}{|x|^{\frac{Q}{p}}{\log}\frac{R}{|x|}}
\right|^{2\frac{p}{2}}dx\right)^{\frac{1}{2}\frac{2}{p}}\left(\int_{\mathbb{G}} |f|^{2\frac{q}{2}}dx\right)^{\frac{1}{2}\frac{2}{q}}
$$

$$\geq\frac{p-1}{p}\left(\int_{\mathbb{G}}\left|\frac{f(f-f_{R})}{|x|^{\frac{Q}{p}}{\log}\frac{R}{|x|}}
\right|^{2}dx\right)^{\frac{1}{2}}=\frac{p-1}{p}\left\|\frac{f(f-f_{R})}{|x|^{\frac{Q}{p}}{\log}\frac{R}{|x|}}\right\|_{L^{2}(\mathbb{G})},$$
where we have used the H\"older inequality in the last line.
This shows \eqref{UP1}. The proof of \eqref{UP2} is similar.
\end{proof}

\section{Another type of critical Hardy inequality}
\label{Sec2}

In this section we prove another type of a critical Hardy inequality. This estimate, or rather its corollary \eqref{CKN}, is analogous to the critical Hardy inequality \eqref{HRn-cr} of Edmunds and Triebel, however, a new feature here is that the logarithmic term enters the right hand side. As such, this result is new also on $\mathbb R^{n}$, however, our techniques allow us to establish it also in the setting of homogeneous groups with arbitrary quasi-norms.

\begin{thm}\label{CritHardy}
	Let $\mathbb{G}$ be a homogeneous group
	of homogeneous dimension $Q$.
	Then for all complex-valued functions $f\in C^{\infty}_{0}(B(0,R)\backslash\{0\}),$
	and any homogeneous quasi-norm $|\cdot|$ on $\mathbb{G}$ we have
	\begin{equation}\label{CritLQHardy}
	\int_{B(0,R)}\frac{|f(x)|^{Q}}{|x|^{Q}}dx\leq Q^{Q}\int_{B(0,R)}|(\log |x|)\mathcal{R}f(x)|^{Q}dx,
	\end{equation}
	where the constant $Q^{Q}$ is optimal. Here $B(0,R)$ is the quasi-ball with respect to the quasi-norm $
	|\cdot|$.
\end{thm}

\begin{rem}
	In the abelian case ${\mathbb G}=(\mathbb R^{n},+)$ we have
	$Q=n$,  so for any quasi-norm $|\cdot|$ on $\mathbb R^{n}$ \eqref{CritLQHardy} implies
	a new inequality with the optimal constant: For each
	$f\in C_{0}^{\infty}(B(0,R)\backslash\{0\}),$
	we have
	\begin{equation}\label{Hardy-r}
	\left\|\frac{f}{|x|}
	\right\|_{L^{n}(B(0,R))}\leq n
	\left\|\left(\log |x|\right)\frac{x}{|x|}\cdot\nabla f\right\|_{L^{n}
		(B(0,R))}.
	\end{equation}
	
	In turn, by using Schwarz's inequality with the standard Euclidean distance $|x|=\sqrt{x^{2}_{1}+\ldots+x^{2}_{n}}$, since $R$ is arbitrary this implies the Euclidean version of the critical Hardy inequality for $\mathbb{G}\equiv\mathbb{R}^{n}$ with the optimal constant:
	\begin{equation}\label{CKN}
	\left\|\frac{f}{|x|}\right\|_{L^{n}(\mathbb{R}^{n})}
	\leq n
	\left\|\log|x|\nabla f\right\|_{L^{n}(\mathbb{R}^{n})},
	\end{equation}
	for all $f\in C_{0}^{\infty}(\mathbb{R}^{n}\backslash\{0\}).$ Here $\nabla$ is the standard
	gradient in $\mathbb{R}^{n}.$
\end{rem}

Thus, even in the abelian case of
$\mathbb{R}^{n}$, for example, the inequality \eqref{Hardy-r} also provides new insights in view
of the arbitrariness of the choice of the not necessarily Euclidean norm.

\begin{proof}[Proof of Theorem \ref{CritHardy}]
	A direct calculation with integrating by parts gives
	\begin{multline*}
	\int_{B(0,R)}\frac{|f(x)|^{p}}{|x|^{Q}}dx
	=\int_{0}^{R}\int_{\wp}|f(\delta_{r}(y))|^{p}r^{Q-1-Q}d\sigma(y)dr\\
	=-p\int_{0}^{R} \log r {\rm Re} \int_{\wp}|f(\delta_{r}(y))|^{p-2} f(\delta_{r}(y)) \overline{\frac{df(\delta_{r}(y))}{dr}}d\sigma(y)dr\\
	\leq p \int_{B(0,R)}\frac{|\mathcal{R}f(x)||f(x)|^{p-1}}
	{|x|^{Q-1}}|\log\,|x||dx=
	p\int_{B(0,R)}
	\frac{|\mathcal{R}f(x)|\log |x|||}{|x|^{\frac{Q}{p}-1}}
	\frac{|f(x)|^{p-1}}{|x|^{\frac{Q(p-1)}{p}}}dx.
	\end{multline*}
	By using H\"{o}lder's inequality, it follows that
	$$
	\int_{B(0,R)}\frac{|f(x)|^{p}}{|x|^{Q}}dx\leq p\left(\int_{B(0,R)}\frac{|\mathcal{R}f(x)|^{p}|\log |x||^{p}}{|x|^{Q-p}}dx\right)
	^{\frac{1}{p}}\left(\int_{B(0,R)}\frac{|f(x)|^{p}}{|x|^{Q}}dx\right)^{\frac{p-1}{p}},
	$$
	which gives \eqref{CritLQHardy}.
	
	Now we show the optimality of the constant, so we need to check the equality
	condition in above H\"older's inequality.
	Let us consider the function
	$$h(x)=\log |x|.$$
	Then a straightforward calculation shows that
	\begin{equation}\label{Holder_eq}
	\left(\frac{|\mathcal{R}h(x)||\log |x||}{|x|^{\frac{Q}{p}-1}}\right)^{p}=
	\left(\frac{|h(x)|^{p-1}}
	{|x|^{\frac{Q (p-1)}{p}}}\right)^{\frac{p}{p-1}},
	\end{equation}
	which satisfies the equality condition in H\"older's inequality.
	This gives the optimality of the constant $Q^{Q}$ in \eqref{CritLQHardy}.
\end{proof}

\begin{cor}[Critical uncertainty principle on a quasi-ball $B(0,R)\subset\mathbb{G}$]\label{Luncertainty}
	Let $\mathbb{G}$ be a homogeneous group of homogeneous dimension
	$Q\geq 2$. Then for each $f\in C^{\infty}_{0}(B(0,R)\backslash\{0\})$ and an arbitrary homogeneous quasi-norm $|\cdot|$ on $\mathbb{G}$ we have
	\begin{equation}\label{UP1}
	\left(\int_{B(0,R)}\left|(\log |x|)\mathcal{R}f\right|^{Q}dx\right)^{\frac{1}{Q}}
	\left(\int_{B(0,R)}|x|^{\frac{Q}{Q-1}}
	|f|^{\frac{Q}{Q-1}}dx\right)^{\frac{Q-1}{Q}}
	\geq\frac{1}{Q}\int_{B(0,R)}
	|f|^{2}dx.
	\end{equation}
\end{cor}
\begin{proof}
	From the inequality \eqref{CritLQHardy} we get
	\begin{multline*}
	\left(\int_{B(0,R)}\left|(\log |x|)\mathcal{R}f\right|^{Q}dx\right)^{\frac{1}{Q}}
	\left(\int_{B(0,R)}|x|^{\frac{Q}{Q-1}}
	|f|^{\frac{Q}{Q-1}}dx\right)^{\frac{Q-1}{Q}}\geq
	\\\frac{1}{Q}\left(\int_{B(0,R)}
	\frac{|f|^{Q}}{|x|^{Q}}
	\,dx\right)^{\frac{1}{Q}}
	\left(\int_{B(0,R)}|x|^{\frac{Q}{Q-1}}
	|f|^{\frac{Q}{Q-1}}dx\right)^{\frac{Q-1}{Q}}
	\\ \geq\frac{1}{Q}\int_{B(0,R)}
	|f|^{2}dx,
	\end{multline*}
	where we have used the H\"older inequality in the last line.
	This shows \eqref{UP1}.
\end{proof}

\section{A class of Hardy-Sobolev type inequalities on quasi-balls}
\label{Sec4}

This section follows the Euclidean ideas from \cite{MOW:Tohoku} which can be modified to suit our setting. In the following inequalities having arbitrary quasi-norms seems interesting.

\begin{thm}\label{oHS}
	Let $\mathbb{G}$ be a homogeneous group
	of homogeneous dimension $Q\geq3$.
	Then for each $f\in C^{\infty}_{0}(B(0,R)\backslash\{0\})$
	and any homogeneous quasi-norm $|\cdot|$ on $\mathbb{G}$ we have
	\begin{equation}\label{thm1-1}
	\left(\int_{B(0,R)}\frac{1}{|x|^{2}}
	\left|f(x)-f\left(\frac{\delta_{R}(x)}{|x|}\right)\right|^{2}dx\right)^{\frac{1}{2}}\leq
	\frac{2}{Q-2}
	\left(\int_{B(0,R)}|\mathcal{R}f|^{2}dx\right)^{\frac{1}{2}},
	\end{equation}
	and
	\begin{multline}\label{thm1-2}
	\left(\int_{B(0,R)}\frac{1}{|x|^{2}}
	\left|f(x)\right|^{2}dx\right)^{\frac{1}{2}}\leq \left(\frac{Q}{Q-2}\right)^{\frac{1}{2}}\frac{1}{R}\left(\int_{B(0,R)}
	\left|f(x)\right|^{2}dx\right)^{\frac{1}{2}}
	\\+
	\frac{2}{Q-2}\left(1+\left(\frac{Q}{Q-2}\right)^{\frac{1}{2}}\right)
	\left(\int_{B(0,R)}|\mathcal{R}f|^{2}dx\right)^{\frac{1}{2}},
	\end{multline}
	where $B(0,R)$ is a quasi-ball with respect to the quasi-norm $|\cdot|$.
\end{thm}

In the case $Q=2$ we have the following inequalities
\begin{thm}\label{Q=2}
	Let $\mathbb{G}$ be a homogeneous group
	of homogeneous dimension $Q=2$.
	Then for each $f\in C^{\infty}_{0}(B(0,R)\backslash\{0\})$
	and any homogeneous quasi-norm $|\cdot|$ on $\mathbb{G}$ we have
	\begin{equation}\label{thm2-1}
	\left(\int_{B(0,R)}\frac{1}{|x|^{2}\left|{\rm log} \frac{R}{|x|}\right|^{2}}
	\left|f(x)-f\left(\frac{\delta_{R}(x)}{|x|}\right)\right|^{2}dx\right)^{\frac{1}{2}}\leq
	2\left(\int_{B(0,R)}|\mathcal{R}f|^{2}dx\right)^{\frac{1}{2}},
	\end{equation}
	and
	\begin{multline}\label{thm2-2}
	\left(\int_{B(0,R)}\frac{\left|f(x)\right|^{2}}{|x|^{2} \left(1+\left|{\rm log} \frac{R}{|x|}\right|^{2}\right)^{2}}
	dx\right)^{\frac{1}{2}}\leq \frac{\sqrt{2}}{R}\left(\int_{B(0,R)}
	\left|f(x)\right|^{2}dx\right)^{\frac{1}{2}}
	\\+ 2\left(1+\sqrt{2}\right)
	\left(\int_{B(0,R)}|\mathcal{R}f|^{2}dx\right)^{\frac{1}{2}}.
	\end{multline}
\end{thm}

\begin{proof}[Proof of Theorem \ref{oHS}]
	
	Introducing polar coordinates $(r,y)=(|x|, 
	\frac{x}{|x|})\in (0,\infty)\times\wp$ on $\mathbb{G}$, 
	where $\wp$ is the quasi-sphere in \eqref{EQ:sphere1}, and using the formula \eqref{EQ:sphere}
	one calculates
	$$
	\int_{B(0,R)}\frac{1}{|x|^{2}}\left|f(x)-f\left(\frac{\delta_{R}(x)}{|x|}\right)\right|^{2}dx$$
	$$=\int_{0}^{R}\int_{\wp}|f(\delta_{r}(y))-f(\delta_{R}(y))|^{2}r^{Q-3}d\sigma(y)dr$$
	$$
	=\frac{1}{Q-2}r^{Q-2}\int_{\wp}|f(\delta_{r}(y))-f(\delta_{R}(y))|^{2}d\sigma(y)\Bigg|_{r=0}^{r=R}
	$$
	$$-\frac{1}{Q-2}\int_{0}^{R}r^{Q-2}\left(\frac{d}{dr}\int_{\wp}|f(\delta_{r}(y))-f(\delta_{R}(y))|^{2}d\sigma(y)\right)dr$$
	$$=-\frac{2}{Q-2}\int_{0}^{R}r^{Q-2}{\rm Re}\int_{\wp}(f(\delta_{r}(y))-f(\delta_{R}(y)))\overline{\frac{df(\delta_{r}(y))}{dr} }d\sigma(y)dr.$$
	Now using Schwarz's inequality, we obtain
	$$
	\int_{B(0,R)}\frac{1}{|x|^{2}}\left|f(x)-f\left(\frac{\delta_{R}(x)}{|x|}\right)\right|^{2}dx$$
	$$\leq\frac{2}{Q-2}\left(\int_{0}^{R}\int_{\wp}|f(\delta_{r}(y))-f(\delta_{R}(y))|^{2}r^{Q-3}d\sigma(y)dr\right)^{\frac{1}{2}}$$
	$$\cdot\left(\int_{0}^{R}\int_{\wp}\left|\frac{df(\delta_{r}(y))}{dr} \right|^{2}r^{Q-1}d\sigma(y)dr\right)^{\frac{1}{2}}$$
	$$=\frac{2}{Q-2}\left(\int_{B(0,R)}\frac{1}{|x|^{2}}\left|f(x)-f\left(\frac{\delta_{R}(x)}{|x|}\right)\right|^{2}dx\right)^{\frac{1}{2}}
	\left(\int_{B(0,R)}|\mathcal{R}f|^{2}dx\right)^{\frac{1}{2}}.$$
	This implies that
	$$
	\left(\int_{B(0,R)}\frac{1}{|x|^{2}}
	\left|f(x)-f\left(\frac{\delta_{R}(x)}{|x|}\right)\right|^{2}dx\right)^{\frac{1}{2}}\leq
	\frac{2}{Q-2}
	\left(\int_{B(0,R)}|\mathcal{R}f|^{2}dx\right)^{\frac{1}{2}},
	$$
	that is, the inequality \eqref{thm1-1} is proved.
	
	To prove \eqref{thm1-2} let us recall the triangle inequality in the form 
	\begin{multline}\label{p01}
	\left(\int_{B(0,R)}\frac{1}{|x|^{2}}|f|^{2}dx\right)^{\frac{1}{2}}=
	\left(\int_{B(0,R)}\frac{1}{|x|^{2}}\left|f(x)-f\left(\frac{\delta_{R}(x)}{|x|}\right)
	+f\left(\frac{\delta_{R}(x)}{|x|}\right)\right|^{2}dx\right)^{\frac{1}{2}}\\
	\leq \left(\int_{B(0,R)}\frac{1}{|x|^{2}}\left|f(x)-f\left(\frac{\delta_{R}(x)}{|x|}\right)\right|^{2}dx\right)^{\frac{1}{2}}
	\\
	+\left(\int_{B(0,R)}\frac{1}{|x|^{2}}\left|f\left(\frac{\delta_{R}(x)}{|x|}\right)\right|^{2}dx\right)^{\frac{1}{2}}.
	\end{multline}
	On the other hand, we obtain
	$$\left(\int_{B(0,R)}\frac{1}{|x|^{2}}\left|f\left(\frac{\delta_{R}(x)}{|x|}\right)\right|^{2}dx\right)^{\frac{1}{2}}
	=\left(\int_{0}^{R}\int_{\wp}|f(\delta_{R}(y))|^{2}r^{Q-3}d\sigma(y)dr\right)^{\frac{1}{2}}$$
	$$=\left(\frac{R^{Q-2}}{Q-2}\int_{\wp}|f(\delta_{R}(y))|^{2}d\sigma(y)\right)^{\frac{1}{2}}$$
	$$=\left(\frac{R^{Q-2}}{Q-2}\frac{Q}{R^{Q}}\int_{0}^{R}\int_{\wp}|f(\delta_{R}(y))|^{2}r^{Q-1}d\sigma(y)dr\right)^{\frac{1}{2}}$$
	$$=\left(\frac{Q}{Q-2}\right)^{\frac{1}{2}}\frac{1}{R}\left(\int_{B(0,R)}\left|f\left(\frac{\delta_{R}(x)}{|x|}\right)\right|^{2}dx\right)^{\frac{1}{2}}$$
	$$\leq \left(\frac{Q}{Q-2}\right)^{\frac{1}{2}}\frac{1}{R}
	\left(\left(\int_{B(0,R)}\left|f\left(\frac{\delta_{R}(x)}{|x|}\right)-f(x)\right|^{2}dx\right)^{\frac{1}{2}}
	+\left(\int_{B(0,R)}|f|^{2}dx\right)^{\frac{1}{2}}\right)$$
	$$\leq \left(\frac{Q}{Q-2}\right)^{\frac{1}{2}}
	\left(\int_{B(0,R)}\frac{1}{|x|^{2}}\left|f\left(\frac{\delta_{R}(x)}{|x|}\right)-f(x)\right|^{2}dx\right)^{\frac{1}{2}}
	$$$$+\left(\frac{Q}{Q-2}\right)^{\frac{1}{2}}\frac{1}{R}\left(\int_{B(0,R)}|f|^{2}dx\right)^{\frac{1}{2}},$$
	that is,
	\begin{multline}\label{p02}
	\left(\int_{B(0,R)}\frac{1}{|x|^{2}}\left|f\left(\frac{\delta_{R}(x)}{|x|}\right)\right|^{2}dx\right)^{\frac{1}{2}}\\ \leq \left(\frac{Q}{Q-2}\right)^{\frac{1}{2}}
	\left(\int_{B(0,R)}\frac{1}{|x|^{2}}\left|f\left(\frac{\delta_{R}(x)}{|x|}\right)-f(x)\right|^{2}dx\right)^{\frac{1}{2}}
	\\+\left(\frac{Q}{Q-2}\right)^{\frac{1}{2}}\frac{1}{R}\left(\int_{B(0,R)}|f|^{2}dx\right)^{\frac{1}{2}}.
	\end{multline}
	Combining \eqref{p02} with  \eqref{p01} we arrive at
	\begin{multline}\label{p03}
	\left(\int_{B(0,R)}\frac{1}{|x|^{2}}|f|^{2}dx\right)^{\frac{1}{2}}
	\leq  \left(1+\left(\frac{Q}{Q-2}\right)^{\frac{1}{2}}\right)\left(\int_{B(0,R)}\frac{1}{|x|^{2}}\left|f(x)-f\left(\frac{\delta_{R}(x)}{|x|}\right)\right|^{2}dx\right)
	^{\frac{1}{2}}
	\\+\left(\frac{Q}{Q-2}\right)^{\frac{1}{2}}\frac{1}{R}\left(\int_{B(0,R)}|f|^{2}dx\right)^{\frac{1}{2}}.
	\end{multline}
	Now using \eqref{thm1-1} we obtain  \eqref{thm1-2}.
\end{proof}

\begin{proof}[Proof of Theorem \ref{Q=2}]
	Introducing polar coordinates $(r,y)=(|x|, \frac{x}{|x|})\in (0,\infty)\times\wp$ on $\mathbb{G}$, where $\wp$ is the quasi-sphere in 
	\eqref{EQ:sphere1}, and using the formula \eqref{EQ:sphere}
	one calculates
	$$
	\int_{B(0,R)}\frac{1}{|x|^{2}|{\rm log}(R/|x|)|^{2}}\left|f(x)-f\left(\frac{\delta_{R}(x)}{|x|}\right)\right|^{2}dx$$
	$$=\int_{0}^{R}\int_{\wp}|f(\delta_{r}(y))-f(\delta_{R}(y))|^{2}\frac{1}{r\left({\rm log}(R/r)\right)^{2}}d\sigma(y)dr$$
	$$
	=\frac{1}{{\rm log}(R/r)}\int_{\wp}|f(\delta_{r}(y))-f(\delta_{R}(y))|^{2}d\sigma(y)\Bigg|_{r=0}^{r=R}
	$$
	$$-\int_{0}^{R}\frac{1}{{\rm log}(R/r)}\left(\frac{d}{dr}\int_{\wp}|f(\delta_{r}(y))-f(\delta_{R}(y))|^{2}d\sigma(y)\right)dr$$
	$$=-2\int_{0}^{R}\frac{1}{{\rm log}(R/r)}{\rm Re}\int_{\wp}(f(\delta_{r}(y))-f(\delta_{R}(y))) \overline{
		\frac{d f(\delta_{r}(y))}{dr} }d\sigma(y)dr.$$
	Here we have used the fact
	$${\rm log}(R/r)={\rm log}\left(1+\left(\frac{R}{r}-1\right)\right)\geq \frac{R}{r}-1=\frac{R-r}{r}.$$
	$$|f(\delta_{r}(y))-f(\delta_{R}(y))|^{2}\leq C|R-r|^{2}.$$
	Using Schwarz's inequality we obtain
	$$
	\int_{B(0,R)}\frac{1}{|x|^{2}|{\rm log}(R/|x|)|^{2}}\left|f(x)-f\left(\frac{\delta_{R}(x)}{|x|}\right)\right|^{2}dx$$
	$$\leq2\left(\int_{0}^{R}\int_{\wp}\frac{1}{r\left({\rm log}(R/r)\right)^{2}}|f(\delta_{r}(y))-f(\delta_{R}(y))|^{2}d\sigma(y)dr\right)^{\frac{1}{2}}$$
	$$\cdot\left(\int_{0}^{R}\int_{\wp}\left|\frac{d f(\delta_{r}(y))}{dr} \right|^{2}rd\sigma(y)dr\right)^{\frac{1}{2}}$$
	$$=2\left(\int_{B(0,R)}\frac{1}{|x|^{2}|{\rm log}(R/|x|)|^{2}}\left|f(x)-f\left(\frac{\delta_{R}(x)}{|x|}\right)\right|^{2}dx\right)^{\frac{1}{2}}
	\left(\int_{B(0,R)}|\mathcal{R}f|^{2}dx\right)^{\frac{1}{2}}.$$
	This proves \eqref{thm2-1}. To prove \eqref{thm2-2}
	we notice
	\begin{multline}\label{p2_1}
	\left(\int_{B(0,R)}\frac{1}{|x|^{2}\left(1+|{\rm log}(R/|x|)|\right)^{2}}|f(x)|^{2}dx\right)^{\frac{1}{2}}
	\\ \leq \left(\int_{B(0,R)}\frac{1}{|x|^{2}\left(1+|{\rm log}(R/|x|)|\right)^{2}}\left|f(x)-f\left(\frac{\delta_{R}(x)}{|x|}\right)\right|^{2}dx\right)^{\frac{1}{2}}
	\\+\left(\int_{B(0,R)}\frac{1}{|x|^{2}\left(1+|{\rm log}(R/|x|)|\right)^{2}}\left|f\left(\frac{\delta_{R}(x)}{|x|}\right)\right|^{2}dx\right)^{\frac{1}{2}}.
	\end{multline}
	
	On the other hand, we have
	
	$$\left(\int_{B(0,R)}\frac{1}{|x|^{2}\left(1+|{\rm log}(R/|x|)|\right)^{2}}\left|f\left(\frac{\delta_{R}(x)}{|x|}\right)\right|^{2}dx\right)^{\frac{1}{2}}
	$$$$=\left(\int_{0}^{R}\int_{\wp}\frac{1}{r\left(1+|{\rm log}(R/r)|\right)^{2}}|f(\delta_{R}(y))|^{2}d\sigma(y)dr\right)^{\frac{1}{2}}$$
	$$=\left(\int_{0}^{R}\frac{1}{r\left(1+|{\rm log}(R/r)|\right)^{2}}dr\int_{\wp}|f(\delta_{R}(y))|^{2}d\sigma(y)\right)^{\frac{1}{2}}$$
	$$=\left(\frac{1}{1+|{\rm log}(R/r)|}\Bigg|_{0}^{R}\int_{\wp}|f(\delta_{R}(y))|^{2}d\sigma(y)\right)^{\frac{1}{2}}$$
	$$=\left(\int_{\wp}|f(\delta_{R}(y))|^{2}d\sigma(y)
	\right)^{\frac{1}{2}}=\left(\frac{2}{R^{2}}
	\int_{0}^{R}\int_{\wp}|f(\delta_{R}(y))|^{2}r
	d\sigma(y)dr\right)^{\frac{1}{2}}$$
	$$=\left(\frac{2}{R^{2}}\right)^{\frac{1}{2}}
	\left(\int_{B(0,R)}\left|f\left(\frac{\delta_{R}(x)}{|x|}
	\right)\right|^{2}dx\right)^{\frac{1}{2}}$$
	$$\leq \frac{\sqrt{2}}{R}\left(\left(\int_{B(0,R)}\left|f\left(\frac{\delta_{R}(x)}{|x|}\right)-f(x)\right|^{2}dx\right)^{\frac{1}{2}}
	+\left(\int_{B(0,R)}|f(x)|^{2}dx\right)^{\frac{1}{2}}\right)$$
	$$\leq \sqrt{2}
	\left(\int_{B(0,R)}\frac{1}{|x|^{2}\left(1+|{\rm log}(R/|x|)|\right)^{2}}\left|f\left(\frac{\delta_{R}(x)}{|x|}\right)-f(x)\right|^{2}dx\right)^{\frac{1}{2}}
	$$$$+\frac{\sqrt{2}}{R}\left(\int_{B(0,R)}|f(x)|^{2}dx\right)^{\frac{1}{2}},$$
	where we have used the fact
	$$\frac{1}{R^{2}}\leq\frac{1}{r^{2}(1+{\rm log}(R/r))^{2}},\,r\in (0,R).$$
	That is
	\begin{multline}\label{p2_2}
	\left(\int_{B(0,R)}\frac{1}{|x|^{2}\left(1+|{\rm log}(R/|x|)|\right)^{2}}\left|f\left(\frac{\delta_{R}(x)}{|x|}\right)\right|^{2}dx\right)^{\frac{1}{2}}
	\\ \leq \sqrt{2}
	\left(\int_{B(0,R)}\frac{1}{|x|^{2}\left(1+|{\rm log}(R/|x|)|\right)^{2}}\left|f\left(\frac{\delta_{R}(x)}{|x|}\right)-f(x)\right|^{2}dx\right)^{\frac{1}{2}}
	\\ +\frac{\sqrt{2}}{R}\left(\int_{B(0,R)}|f(x)|^{2}dx\right)^{\frac{1}{2}}.
	\end{multline}
	Combining \eqref{p2_2} with  \eqref{p2_1} we arrive at
	\begin{multline}
	\left(\int_{B(0,R)}\frac{1}{|x|^{2}\left(1+|{\rm log}(R/|x|)|\right)^{2}}|f(x)|^{2}dx\right)^{\frac{1}{2}}
	\\ \leq (1+\sqrt{2})\left(\int_{B(0,R)}\frac{1}{|x|^{2}\left(1+|{\rm log}(R/|x|)|\right)^{2}}\left|f(x)-f\left(\frac{\delta_{R}(x)}{|x|}\right)\right|^{2}dx\right)^{\frac{1}{2}}
	\\ +\frac{\sqrt{2}}{R}\left(\int_{B(0,R)}|f(x)|^{2}dx\right)^{\frac{1}{2}}.
	\end{multline}
	Finally, using \eqref{thm2-1} we obtain  \eqref{thm2-2}.
	
\end{proof}

\begin{thebibliography}{MOW15b}

\bibitem[AS06]{Adimurthi-Sekar}
Adimurthi and A.~Sekar.
\newblock Role of the fundamental solution in {H}ardy-{S}obolev-type
  inequalities.
\newblock {\em Proc. Roy. Soc. Edinburgh Sect. A}, 136(6):1111--1130, 2006.

\bibitem[BEL15]{BEL}
A.~A. Balinsky, W.~D. Evans, and R.~T. Lewis.
\newblock {\em The analysis and geometry of {H}ardy's inequality}.
\newblock Universitext. Springer, Cham, 2015.

\bibitem[BL85]{Brez1}
H. Brezis and E. Lieb.
\newblock Sobolev inequalities with remainder terms.
\newblock {\em J. Funct. Anal.}, 62:73--86, 1985.

\bibitem[BM97]{Brez2}
H. Brezis and M. Marcus.
\newblock Hardy's inequalities revisited.
\newblock {\em Ann. Scuola
	Norm. Sup. Pisa Cl. Sci.}, 25(4):217--237, 1997.

\bibitem[BN83]{Brez3}
H. Br\'ezis and L. Nirenberg.
Positive solutions of nonlinear elliptic equations involving critical Sobolev exponents. 
{\em Comm. Pure Appl. Math.}, 36(4): 437--477, 1983. 

\bibitem[BV97]{Brez4}
H. Brezis and J. V\'azquez. 
Blow-up solutions of some nonlinear elliptic problems. 
{\em Rev. Mat. Univ. Complut. Madrid},  10(2): 443--469, 1997.

\bibitem[CCR15]{Ciatti-Cowling-Ricci}
P.~Ciatti, M.~G. Cowling, and F.~Ricci.
\newblock Hardy and uncertainty inequalities on stratified {L}ie groups.
\newblock {\em Adv. Math.}, 277:365--387, 2015.


\bibitem[CF02]{Cianchi-Fusco}
A.~Cianchi, N.~Fusco.
\newblock Functions of bounded variation and rearrangements.
\newblock {\em Archive for Rational Mechanics and Analysis}, 165(1):1--40, 2002.

\bibitem[CF08]{Cianchi-Ferone08}
A.~Cianchi, N.~Ferone.  
\newblock Hardy inequalities with non-standard remainder terms.
\newblock {\em Annales de l'Institut Henri Poincar{\'e}. Analyse Non Lin{\'e}aire}, 25(5):889--906, 2008.

\bibitem[CF08]{Cianchi-Ferone08}
A.~Cianchi, N.~Ferone.  
\newblock Best remainder norms in Sobolev-Hardy inequalities.
\newblock {\em Indiana University Mathematics Journal}, 58(3):1051-- 1096, 2009.

\bibitem[D'A05]{DAmbrosio-Hardy}
L.~D'Ambrosio.
\newblock Hardy-type inequalities related to degenerate elliptic differential
  operators.
\newblock {\em Ann. Sc. Norm. Super. Pisa Cl. Sci. (5)}, 4(3):451--486, 2005.

\bibitem[Dav99]{Davies}
E.~B. Davies.
\newblock A review of {H}ardy inequalities.
\newblock In {\em The {M}az'ya anniversary collection, {V}ol. 2 ({R}ostock,
  1998)}, volume 110 of {\em Oper. Theory Adv. Appl.}, pages 55--67.
  Birkh\"auser, Basel, 1999.

\bibitem[DGP11]{DGP-Hardy-potanal}
D.~Danielli, N.~Garofalo, and N.~C. Phuc.
\newblock Hardy-{S}obolev type inequalities with sharp constants in
  {C}arnot-{C}arath{\'e}odory spaces.
\newblock {\em Potential Anal.}, 34(3):223--242, 2011.

\bibitem[Dye70]{Dyer-1970}
J.~L. Dyer.
\newblock A nilpotent {L}ie algebra with nilpotent automorphism group.
\newblock {\em Bull. Amer. Math. Soc.}, 76:52--56, 1970.

\bibitem[ET99]{ET-1999}
D.~E. Edmunds and H.~Triebel.
\newblock Sharp {S}obolev embeddings and related {H}ardy inequalities: the
  critical case.
\newblock {\em Math. Nachr.}, 207:79--92, 1999.

\bibitem[Fol75]{Folland-FS}
G.~B. Folland.
\newblock Subelliptic estimates and function spaces on nilpotent {L}ie groups.
\newblock {\em Ark. Mat.}, 13(2):161--207, 1975.

\bibitem[FR16]{FR}
V.~Fischer and M.~Ruzhansky.
\newblock {\em Quantization on nilpotent Lie groups}, volume 314 of {\em
  Progress in Mathematics}.
\newblock Birkh\"auser, 2016. (open access book)

\bibitem[FS82]{FS-Hardy}
G.~B. Folland and E.~M. Stein.
\newblock {\em Hardy spaces on homogeneous groups}, volume~28 of {\em
  Mathematical Notes}.
\newblock Princeton University Press, Princeton, N.J.; University of Tokyo
  Press, Tokyo, 1982.

\bibitem[GK08]{GolKom}
J.~A. Goldstein and I.~Kombe.
\newblock The {H}ardy inequality and nonlinear parabolic equations on {C}arnot
  groups.
\newblock {\em Nonlinear Anal.}, 69(12):4643--4653, 2008.

\bibitem[GL90]{GL}
N.~Garofalo and E.~Lanconelli.
\newblock Frequency functions on the {H}eisenberg group, the uncertainty
  principle and unique continuation.
\newblock {\em Ann. Inst. Fourier (Grenoble)}, 40(2):313--356, 1990.

\bibitem[Har20]{Hardy:1920}
G.~H. Hardy.
\newblock Note on a theorem of {H}ilbert.
\newblock {\em Math. Z.}, 6(3-4):314--317, 1920.

\bibitem[II15]{II}
N.~Ioku and M.~Ishiwata.
\newblock A scale invariant form of a critical {H}ardy inequality.
\newblock {\em Int. Math. Res. Not. IMRN}, (18):8830--8846, 2015.

\bibitem[IIO16]{IIO}
N.~Ioku, M.~Ishiwata, and T.~Ozawa.
\newblock Sharp remainder of a critical {H}ardy inequality.
\newblock {\em Arch. Math. (Basel)}, 106(1):65--71, 2016.

\bibitem[JS11]{Jin-Shen:Hardy-Rellich-AM-2011}
Y.~Jin and S.~Shen.
\newblock Weighted {H}ardy and {R}ellich inequality on {C}arnot groups.
\newblock {\em Arch. Math. (Basel)}, 96(3):263--271, 2011.

\bibitem[Lia13]{Lian:Rellich}
B.~Lian.
\newblock Some sharp {R}ellich type inequalities on nilpotent groups and
  application.
\newblock {\em Acta Math. Sci. Ser. B Engl. Ed.}, 33(1):59--74, 2013.

\bibitem[MOW13]{MOW:Tohoku}
S.~Machihara, T.~Ozawa, and H.~Wadade.
\newblock {H}ardy type inequalities on balls.
\newblock {\em Tohoku Math. J.}, 65:321--330, 2013.

\bibitem[MOW15a]{MOW:Hardy-Hayashi}
S.~Machihara, T.~Ozawa, and H.~Wadade.
\newblock On the {H}ardy type inequalities.
\newblock {\em preprint}, 2015.

\bibitem[MOW15b]{MOW:Sobolev-Lorenz-Zygmund}
S.~Machihara, T.~Ozawa, and H.~Wadade.
\newblock Scaling invariant {H}ardy inequalities of multiple logarithmic type
  on the whole space.
\newblock {\em J. Inequal. Appl.}, pages 2015:281, 13, 2015.

\bibitem[NZW01]{NZW-Hardy-p}
P.~Niu, H.~Zhang, and Y.~Wang.
\newblock Hardy type and {R}ellich type inequalities on the {H}eisenberg group.
\newblock {\em Proc. Amer. Math. Soc.}, 129(12):3623--3630, 2001.

\bibitem[RS15]{Ruzhansky-Suragan:Layers}
M.~Ruzhansky and D.~Suragan.
\newblock Layer potentials, {K}ac's problem, and refined {H}ardy inequality on
  homogeneous {C}arnot groups.
\newblock {\em arXiv:1512.02547}, 2015.

\bibitem[RS16a]{Ruzhansky-Suragan:Hardy-hom-Lp}
M.~Ruzhansky and D.~Suragan.
\newblock Hardy and Rellich inequalities, identities, and sharp remainders on homogeneous groups.
\newblock {\em arXiv:1603.06239}, 2016.

\bibitem[RS16b]{Ruzhansky-Suragan:L2-CKN}
M.~Ruzhansky and D.~Suragan.
\newblock Anisotropic $L^{2}$-weighted Hardy and $L^{2}$-Caffarelli-Kohn-Nirenberg inequalities.
\newblock {\em Commun. Contemp. Math.}, to appear, 2016. {\em arXiv:1610.07032}

\bibitem[RS16c]{Ruzhansky-Suragan:horizontal}
M.~Ruzhansky and D.~Suragan.
\newblock On horizontal Hardy, Rellich, Caffarelli-Kohn-Nirenberg and $p$-sub-Laplacian inequalities on stratified groups.
\newblock {\em J. Differential Equations}, to appear, 2016. {\em arXiv:1605.06389}

\bibitem[RS16d]{Ruzhansky-Suragan:squares}
M.~Ruzhansky and D.~Suragan.
\newblock Local Hardy and Rellich inequalities for sums of squares of vector fields.
{\em Adv. Differ. Equations}, to appear, 2016.
\newblock {\em arXiv:1601.06157}.

\end{thebibliography}

\end{document}